\documentclass[a4paper,11pt]{elsarticle}
\usepackage[utf8]{inputenc}
\usepackage[T1]{fontenc}
\usepackage{amsmath,amssymb,amsthm}
\usepackage{geometry}
\usepackage{hyperref}

\newtheorem{theorem}{Theorem}[section]
\newtheorem{lemma}[theorem]{Lemma}
\newtheorem{corollary}[theorem]{Corollary}

\theoremstyle{definition}
\newtheorem{definition}[theorem]{Definition}
\newtheorem{example}[theorem]{Example}
\theoremstyle{remark}

\DeclareMathOperator{\End}{End}

\DeclareMathOperator{\im}{Im}
\newcommand{\m}{\mathfrak{m}}

\begin{document}

\begin{frontmatter}
\title{Rota--Baxter Operators on Truncated Polynomial Algebras}

\author{Azhar Farooq }
\address{Abdus Salam School of Mathematical Sciences,\\ GC University, Lahore-Pakistan\\e-mail:
			azhar.farooq@sms.edu.pk
		}

\begin{abstract}
Let $K$ be a field of characteristic zero and let  $\m = (x_1,\dots,x_n)$ be a maximal ideal of the polynomial ring $R = K[x_1,\dots,x_n]$. We classify all Rota--Baxter operators of weight zero and weight one on the truncated polynomial algebra $R = K[x_1,\dots,x_n]/\m^2$.
For weight zero we show that such operators are exactly the linear maps $P$ satisfying $P^2 = 0$ and $\im(P) \subseteq \m/\m^2$.
For non‑zero weight a standard rescaling reduces the problem to weight one.
Weight one operators split into two disjoint families determined by the scalar $P(1) \in \{0, -1\}$; on the maximal ideal they act as an endomorphism $L$ with $L^2 + L = 0$, i.e. $-L$ is an idempotent.
Each family is isomorphic to the variety of idempotent matrices.
\end{abstract}

\end{frontmatter}

\section{Introduction}

Rota--Baxter operators were first introduced by Glen Baxter in the context of probability theory \cite{Baxter} and later studied by Gian-Carlo Rota and others as algebraic tools in combinatorics and mathematical physics \cite{Rota}.  Recall that for an associative $K$-algebra $A$ and a scalar $\lambda\in K$, a {\em Rota--Baxter operator of weight $\lambda$} is a $K$-linear map $P:A\to A$ satisfying the identity 
\begin{equation}\label{RBdef}
    P(x)P(y) \;=\; P\bigl(xP(y)+P(x)y + \lambda\,xy\bigr)
\end{equation}
for all $x,y\in A$.  Classic examples include the integral operator (of weight~0) on continuous functions, whose integration-by-parts formula encodes the Rota--Baxter relation.  The case $\lambda=0$ and $\lambda=1$ are of special interest, since any nonzero weight can be rescaled to~1 \cite{Guo}.  Rota--Baxter algebras have rich connections to combinatorics and physics: for instance, the classical Spitzer identity in probability theory is naturally derived using a Rota--Baxter structure, and the Connes--Kreimer theory of renormalization in perturbative quantum field theory revealed that the decomposition of Feynman graphs is governed by a Rota--Baxter operator \cite{ConnesKreimer2000,ConnesKreimer1998,EG2007}.  There are also links to dendriform and shuffle algebras \cite{EG2008,Aguiar2000,GuoKeigher2000}, as well as to solutions of the classical Yang--Baxter equation.

The structure and classification of Rota--Baxter operators have been extensively studied on various algebras (see \cite{Guo} for a survey).  In particular, one encounters Rota--Baxter operators on polynomial algebras, free (noncommutative) algebras, and matrix or semigroup algebras, among others \cite{ZhengGuoRosenkranz2015,GuoZheng2021,Cartier1972,ZhengGuo2019,Gubarev2018,Tang2021,Gubarev2025}.  In this paper we initiate the study of Rota--Baxter operators on {\em truncated polynomial algebras}.  Concretely, let 
\[
R \;=\; K[x_1,\dots,x_n]/\m^2,\qquad \m=(x_1,\dots,x_n),
\]
where $K$ is a field of characteristic zero.  Then $R$ decomposes as $K\oplus V$ with $V=\m/\m^2\cong K^n$ and $V^2=0$.  Algebras of the form $K\oplus V$ with $V^2=0$ play a fundamental role in deformation theory and geometry, encoding first-order neighbourhoods of points.  We classify all Rota--Baxter operators of weight $0$ and $1$ on $R$.

Our main results can be summarized as follows:
\begin{itemize}
\item {\bf Weight 0.}  We prove that a $K$-linear map $P:R\to R$ is a Rota--Baxter operator of weight $0$ if and only if $P^2=0$ and $\im(P)\subseteq V$ (Theorem~\ref{thm:wt0}).  Equivalently, choosing the basis $(1,\bar x_1,\dots,\bar x_n)$ of $R$, the matrix of $P$ has the block form 
\[
\begin{pmatrix} 0 & 0 \\ v_0 & L \end{pmatrix},
\]
where $v_0\in V$ and $L\in\End_K(V)$ satisfy $L^2=0$ and $L\,v_0=0$.  In other words, $P$ projects into the nilpotent ideal $V$ and is itself nilpotent.

\item {\bf Weight 1.}  For a nonzero weight operator we rescale to $\lambda=1$ (by Lemma~\ref{lem:rescale}).  We show that any Rota--Baxter operator $P$ of weight $1$ on $R$ must satisfy $P(1)\in\{0,-1\}$ (Theorem~\ref{thm:wt1}).  Moreover, in either case $P$ induces an endomorphism $L=P|_V$ on $V$ satisfying $L^2+L=0$ (so $-L$ is idempotent).  Conversely, given any scalar $\lambda\in\{0,-1\}$ and any idempotent $Q\in \End_K(V)$, the formula
\[
P(a+u)\;=\;\lambda\,a \;-\;Q(u)\qquad (a\in K,\ u\in V)
\]
defines a Rota--Baxter operator of weight~1 on $R$.  Thus the set of weight~1 operators splits into two disjoint families (depending on $\lambda=P(1)$), each parametrized by idempotent matrices.  In particular, the variety of weight-1 operators is isomorphic to the disjoint union of two copies of the variety of idempotent $n\times n$ matrices (Corollary~\ref{cor:wt1geo}).

\end{itemize}

The remainder of the paper is organized as follows.  In Section~\ref{sec:prelim} we review basic properties of the truncated polynomial algebra $R$.  Section~\ref{sec:wt0} contains the proof of the weight~0 classification, and Section~\ref{sec:wt1} treats the weight~1 case. Our arguments are essentially elementary linear algebra once one writes linear maps on $R$ in the basis $(1,\bar x_1,\dots,\bar x_n)$.

\section{Preliminaries}\label{sec:prelim}

Throughout, $K$ is a field of characteristic zero.
For $n \ge 1$ we set $S = K[x_1,\dots,x_n]$, $\m = (x_1,\dots,x_n)$, and
\[
R = S / \m^2.
\]
Denote by $\bar x_i$ the image of $x_i$ in $R$.  Every element of $R$ can be written uniquely as
\[
a + \sum_{i=1}^n b_i \bar x_i, \qquad a, b_i \in K.
\]
Thus $R = K \oplus V$ where $V = \m/\m^2$ is the $n$-dimensional vector space with basis $\bar x_1,\dots,\bar x_n$.  Multiplication satisfies
\[
(a+u)(b+v) = ab + av + bu, \qquad a,b \in K,\; u,v \in V,
\]
so $V \cdot V = 0$. We identify $K$ with the subring of scalar multiples of the unit $1$.

For a $K$-vector space $W$, $\End_K(W)$ denotes the algebra of $K$-linear endomorphisms of $W$.
A projection on $V$ is an idempotent $Q \in \End_K(V)$, i.e. $Q^2 = Q$.

\section{Weight zero Rota--Baxter operators}\label{sec:wt0}

In this section we classify all Rota--Baxter operators of weight $\lambda = 0$ on $R$.

\begin{definition}
A $K$-linear map $P : R \to R$ is a \emph{Rota--Baxter operator of weight zero} if
\begin{equation}\label{RB0}
P(x)P(y) = P\bigl(x P(y) + P(x) y\bigr) \qquad \forall x,y \in R.
\end{equation}
\end{definition}

\begin{lemma}\label{lem:im_in_V}
If $P$ satisfies \eqref{RB0}, then $\im(P) \subseteq V$.  In particular $P(1) \in V$ and $P(V) \subseteq V$.
\end{lemma}

\begin{proof}
Write $P(1) = \alpha + v_0$ with $\alpha \in K$, $v_0 \in V$.
Using $x = y = 1$ in \eqref{RB0}:
\[
P(1)^2 = P(2P(1)) = 2P(P(1)).
\]
Expanding the left side gives $\alpha^2 + 2\alpha v_0$, while the right side is $2\alpha P(1) + 2P(v_0) = 2\alpha^2 + 2\alpha v_0 + 2P(v_0)$.
Cancelling $2\alpha v_0$ we obtain
\begin{equation}\label{eq3}
    \alpha^2 = 2\alpha^2 + 2P(v_0) \quad\Longrightarrow\quad 0 = \alpha^2 + 2P(v_0).
\end{equation}
Thus $2P(v_0) = -\alpha^2 \in K$, so $P(v_0)$ is a scalar.

Now use $x = v_0,\; y = 1$:
\[
P(v_0)P(1) = P\bigl(v_0 P(1) + P(v_0)\bigr).
\]
Since $P(v_0) = c \in K$, the left side is $c(\alpha + v_0) = c\alpha + c v_0$.
On the right, $v_0 P(1) + c = v_0(\alpha + v_0) + c = \alpha v_0 + c$ because $v_0^2=0$.
Applying $P$ yields $\alpha P(v_0) + P(c) = \alpha c + c P(1) = \alpha c + c(\alpha + v_0) = 2\alpha c + c v_0$.
Equating the two sides gives $c\alpha + c v_0 = 2\alpha c + c v_0$, hence $c\alpha = 2\alpha c$, i.e. $\alpha c = 0$ in $K$.
So either $\alpha = 0$ or $c = 0$.  If $c = 0$, then from $c = -\alpha^2/2$ (from \ref{eq3} after dividing by $2$) we obtain $\alpha^2 = 0$, hence $\alpha = 0$ anyway.
Thus $\alpha = 0$ and then \ref{eq3} gives $2P(v_0) = 0$, so $P(v_0) = 0$.
Therefore $P(1) = v_0 \in V$ and $P(P(1)) = 0$.

Now for any $u,v \in V$, write $P(u) = \beta + u'$, $P(v) = \gamma + v'$ with $\beta,\gamma \in K$, $u',v' \in V$.
Using \eqref{RB0} with $x = u, y = v$:
\[
P(u)P(v) = \beta\gamma + \beta v' + \gamma u',
\]
while the right side is $P\bigl(u P(v) + P(u) v\bigr) = P(\beta v + \gamma u) = \beta P(v) + \gamma P(u) = 2\beta\gamma + \beta v' + \gamma u'$.
Hence $\beta\gamma = 2\beta\gamma$, so $\beta\gamma = 0$.  Taking $u = v$ gives $\beta^2 = 0$, so $\beta = 0$.
Thus $P(V) \subseteq V$.
Since every $x \in R$ is $a + u$ with $a \in K$, $u \in V$, we have $P(x) = aP(1) + P(u) \in V$.
\end{proof}

\begin{lemma}\label{lem:P2zero}
If $P$ satisfies \eqref{RB0}, then $P^2 = 0$.
\end{lemma}

\begin{proof}
We already know $P(P(1)) = 0$ from the previous proof.  For any $v \in V$, use $x = v$, $y = 1$ in \eqref{RB0}:
\[
P(v)P(1) = P\bigl(v P(1) + P(v)\bigr).
\]
Since $P(1), P(v) \in V$ and $V^2 = 0$, the left side is $0$.  On the right, $v P(1) = 0$, so the argument of $P$ is just $P(v)$.  Hence $0 = P(P(v))$, so $P^2(v) = 0$.  Together with $P^2(1) = 0$, linearity gives $P^2 = 0$.
\end{proof}

\begin{theorem}\label{thm:wt0}
Let $R = S/\m^2$ be as above.  A $K$-linear map $P : R \to R$ is a Rota--Baxter operator of weight zero if and only if $\im(P) \subseteq V$ and $P^2 = 0$.
\end{theorem}

\begin{proof}
Necessity is Lemmas~\ref{lem:im_in_V} and~\ref{lem:P2zero}.
For sufficiency, assume $\im(P) \subseteq V$ and $P^2 = 0$.  For any $x,y \in R$, write them as $x = a+u$, $y = b+v$.  Then $P(x), P(y) \in V$, so their product is $0$.
On the other hand,
\[
x P(y) + P(x) y = (a+u)P(y) + P(x)(b+v) = a P(y) + b P(x)
\]
because $u P(y) = P(x) v = 0$ (product in $V$).
Applying $P$ gives $a P^2(y) + b P^2(x) = 0$.  Hence the Rota--Baxter identity holds.
\end{proof}

\subsection{Matrix form and parameterisation}
With respect to the basis $1, \bar x_1,\dots,\bar x_n$ of $R$, any linear map $P$ with $\im(P) \subseteq V$ has a matrix of the form
\begin{equation}\label{Pmatrix}
P = \begin{pmatrix}
0 & 0 \\[2pt]
v_0 & L
\end{pmatrix},
\end{equation}
where $v_0 \in K^n$ is the column vector $(p_{10},\dots,p_{n0})^T$ and $L = (p_{ij})_{1 \le i,j \le n} \in M_n(K)$.
The condition $P^2 = 0$ translates to
\[
L^2 = 0 \qquad\text{and}\qquad L v_0 = 0.
\]
Thus the set of weight‑zero Rota--Baxter operators is in bijection with
\[
\mathcal{RB}_0 = \{(v_0, L) \in V \times \End_K(V) \mid L^2 = 0,\; L v_0 = 0\}.
\]

\begin{example}\label{ex:n3}
For $n=3$, the matrix of a weight‑zero Rota--Baxter operator is
\[
\begin{pmatrix}
0 & 0 & 0 & 0\\
a_1 & \ell_{11} & \ell_{12} & \ell_{13}\\
a_2 & \ell_{21} & \ell_{22} & \ell_{23}\\
a_3 & \ell_{31} & \ell_{32} & \ell_{33}
\end{pmatrix}
\]
with $L^2 = 0$ and $L (a_1,a_2,a_3)^T = 0$.
\end{example}

\section{Weight one Rota--Baxter operators}\label{sec:wt1}

For a Rota--Baxter operator with $\lambda \neq 0$, the following simple rescaling recovers the weight‑one case.

\begin{lemma}\cite{Guo}\label{lem:rescale}
Let $P$ be a Rota--Baxter operator of weight $\lambda \neq 0$ on an algebra $A$.  Then $Q = \lambda^{-1} P$ is a Rota--Baxter operator of weight $1$.
\end{lemma}
Hence the classification of Rota--Baxter operators of any non‑zero weight immediately follows from that of weight one.
We now classify Rota--Baxter operators of weight $\lambda = 1$ on $R$.

\begin{lemma}\label{lem:wt1_form}
If $P : R \to R$ is a Rota--Baxter operator of weight one, then $P|_V$ maps $V$ to $V$, and  $P(1) \in \{0, -1\}$.  Moreover, for all $u \in V$, $P(u) = L(u)$ with $L \in \End_K(V)$ satisfying $L^2 + L = 0$.
\end{lemma}

\begin{proof}
Write $P(1) = \alpha + v_0$ with $\alpha \in K$, $v_0 \in V$, and for $u \in V$ write $P(u) = f(u) + L(u)$ where $f : V \to K$ is linear and $L \in \End_K(V)$.

First take $x, y \in V$.  The weight‑one identity \eqref{RBdef} with $\lambda=1$ gives
\[
P(x)P(y) = P\bigl(x P(y) + P(x) y + xy\bigr).
\]
Since $V^2=0$, we have $xy=0$, $x L(y)=0$, and $L(x) y=0$.  Then
\begin{align*}
\text{LHS} &= (f(x)+L(x))(f(y)+L(y)) = f(x)f(y) + f(x)L(y) + f(y)L(x),\\
\text{RHS} &= P\bigl(f(y)x + f(x)y\bigr) = f(y)P(x) + f(x)P(y)\\
 &= 2f(x)f(y) + f(y)L(x) + f(x)L(y).
\end{align*}
Equating scalar parts yields $f(x)f(y) = 2f(x)f(y)$, hence $f(x)f(y)=0$ for all $x,y$.  A non‑zero linear functional would have $f(u_0)=1$ for some $u_0$, then $f(u_0)^2=1$, contradiction.  Thus $f=0$ and $P(V) \subseteq V$.

Now set $x = y = 1$ in \eqref{RBdef}:
\[
P(1)^2 = P(2P(1) + 1).
\]
Using $P(1)=\alpha+v_0$ and $P|_V = L$,
\begin{align*}
\text{LHS} &= \alpha^2 + 2\alpha v_0,\\
\text{RHS} &= P(2\alpha + 2v_0 + 1) = (2\alpha+1)P(1) + 2P(v_0) \\
&= (2\alpha+1)(\alpha+v_0) + 2L(v_0) = 2\alpha^2+\alpha + (2\alpha+1)v_0 + 2L(v_0).
\end{align*}
Comparing components gives
\begin{align*}
\alpha^2 &= 2\alpha^2 + \alpha \;\Longrightarrow\; \alpha(\alpha+1)=0,\\
2\alpha v_0 &= (2\alpha+1)v_0 + 2L(v_0) \;\Longrightarrow\; L(v_0) = -\tfrac12 v_0. \tag{B}
\end{align*}
So $\alpha \in \{0, -1\}$.

Next take $x = u \in V$, $y = 1$:
\[
P(u)P(1) = P\bigl(u P(1) + P(u) + u\bigr).
\]
Compute:
\begin{align*}
\text{LHS} &= L(u)(\alpha+v_0) = \alpha L(u),\\
\text{RHS} &= P\bigl(u(\alpha+v_0) + L(u) + u\bigr) = P\bigl(\alpha u + L(u) + u\bigr)\\
&= (\alpha+1)L(u) + L^2(u).
\end{align*}
Hence $\alpha L(u) = (\alpha+1)L(u) + L^2(u)$, i.e. $L^2 + L = 0$ on $V$.

Finally, apply $L$ to Equation (B):
\[
L^2(v_0) = -\tfrac12 L(v_0) = \tfrac14 v_0.
\]
But $L^2 = -L$, so $L^2(v_0) = -L(v_0) = \tfrac12 v_0$.  Thus $\tfrac14 v_0 = \tfrac12 v_0$, forcing $v_0=0$.  Hence $P(1) = \alpha \in \{0,-1\}$.
\end{proof}

\begin{theorem}\label{thm:wt1}
A $K$-linear map $P : R \to R$ is a Rota--Baxter operator of weight one if and only if there exist $\alpha \in \{0,-1\}$ and an idempotent $Q \in \End_K(V)$ (i.e. $Q^2 = Q$) such that
\[
P(a+u) = \alpha a - Q(u) \qquad \forall a \in K,\; u \in V.
\]
\end{theorem}

\begin{proof}
From Lemma~\ref{lem:wt1_form} we have $P(1) = \alpha \in \{0,-1\}$ and $P|_V = L$ with $L^2+L=0$.  Define $Q = -L$; then $Q^2 = (-L)^2 = L^2 = -L = Q$, so $Q$ is idempotent.  Conversely, given $\alpha \in \{0,-1\}$ and an idempotent $Q$, set $P(a+u) = \alpha a - Q(u)$.  One verifies the weight‑one identity by a direct computation similar to the weight‑zero case, using that $V^2=0$ and the property $Q^2 = Q$.
\end{proof}

\begin{corollary}\label{cor:wt1geo}
The set $\mathcal{RB}_1$ of weight‑one Rota--Baxter operators on $R$ is isomorphic to the disjoint union of two copies of the variety of idempotent matrices in $M_n(K)$:
\[
\mathcal{RB}_1 \cong \bigl(\{0\} \times \mathcal{P}\bigr) \sqcup \bigl(\{-1\} \times \mathcal{P}\bigr),
\]
where $\mathcal{P} = \{ Q \in M_n(K) \mid Q^2 = Q\}$.  
\end{corollary}

\begin{proof}
The isomorphism is immediate from Theorem~\ref{thm:wt1}.
\end{proof}


\end{document}